\documentclass[reqno, 11pt]{amsart}
\usepackage[colorinlistoftodos]{todonotes}
\usepackage{url}
\usepackage[colorlinks, linkcolor=blue, citecolor=blue, urlcolor=blue]{hyperref}

\usepackage{verbatim}
\usepackage{float}

\usepackage{mdframed}

\binoppenalty=\maxdimen
\relpenalty=\maxdimen

\usepackage{times}

\usepackage{amsmath}
\usepackage{amsthm}
\usepackage{amssymb}
\usepackage{mathrsfs}
\usepackage{amsxtra}
\usepackage{pifont}
\usepackage{amsbsy}

\usepackage{pb-diagram}
\usepackage{lamsarrow}
\usepackage{pb-lams}

\numberwithin{equation}{section}
\numberwithin{table}{section}
\numberwithin{figure}{section}

\theoremstyle{plain}
\newtheorem{theorem}{Theorem}[section]
\newtheorem{lemma}[theorem]{Lemma}

\newtheorem{proposition}[theorem]{Proposition}

\theoremstyle{definition}

\numberwithin{theorem}{section}

\usepackage{bbm}
\usepackage{euscript}

\usepackage{graphicx}
\usepackage{epstopdf}

\newskip\aline \newskip\halfaline
\newcommand{\cond}{\,\Vert\,}

\newcommand*\xbar[1]{%
   \hbox{%
     \vbox{%
       \hrule height 0.5pt % The actual bar
       \kern0.5ex%         % Distance between bar and symbol
       \hbox{%
         \kern-0.1em%      % Shortening on the left side
         \ensuremath{#1}%
         \kern-0.1em%      % Shortening on the right side
       }%
     }%
   }%
}

\DeclareFontFamily{OT1}{pzc}{}
\DeclareFontShape{OT1}{pzc}{m}{it}{<-> s * [1.10] pzcmi7t}{}
\DeclareMathAlphabet{\mathpzc}{OT1}{pzc}{m}{it}

\newcommand{\one}{{\mathbbm{1}}}

\newcommand\bE{{\mathbb E}}

\newcommand{\bP}{{{\mathbb P}}}

\newcommand\bR{{\mathbb R}}

\newcommand{\be}{{\boldsymbol{e}}}
\newcommand{\bsf}{\boldsymbol{f}}

\newcommand{\bn}{\boldsymbol{n}}
\newcommand{\bp}{{\boldsymbol{p}}}

\newcommand{\bx}{{\boldsymbol{x}}}

\newcommand{\bsD}{\boldsymbol{D}}

\newcommand{\bsN}{\boldsymbol{N}}

\newcommand{\bsX}{{\boldsymbol{X}}}
\newcommand{\bsY}{{\boldsymbol{Y}}}
\newcommand{\bsZ}{\boldsymbol{Z}}

\newcommand{\bgamma}{{\boldsymbol{\gamma}}}

\newcommand{\bGamma}{\boldsymbol{\Gamma}}

\newcommand{\bsi}{\boldsymbol{\sigma}}

\newcommand{\si}{{\sigma}}

\newcommand{\vfi}{{\varphi}}

\newcommand{\eD}{\EuScript{D}}

\newcommand{\eK}{\EuScript{K}}

\newcommand{\eS}{\EuScript{S}}

\DeclareMathOperator{\tr}{{\rm tr}}

\DeclareMathOperator{\vol}{vol}

\DeclareMathOperator{\diag}{Diag}

\DeclareMathOperator{\Sym}{\mathbf{Sym}}
\DeclareMathOperator{\Spec}{Spec}

\DeclareMathOperator{\Hess}{Hess}

\DeclareMathOperator{\GOE}{GOE}

\DeclareMathOperator{\Proj}{Proj}

\DeclareMathOperator{\var}{Var}
\DeclareMathOperator{\cov}{Cov}
\DeclareMathOperator{\cor}{Corr}

\DeclareMathOperator{\Var}{Var}

\DeclareMathOperator{\mult}{mult}

\newcommand{\ra}{\rightarrow}

\newcommand{\lan}{\langle}
\newcommand{\ran}{\rangle}

\newcommand{\rb}{\,\big]}
\newcommand{\lb}{\big[\,}
\newcommand{\rbr}{\,\big\}}
\newcommand{\lbr}{\big\{\,}

\newcommand{\rp}{\,\big)}
\newcommand{\lp}{\big(\,}
\newcommand{\lv}{\big\vert\,}
\newcommand{\rv}{\,\big\vert}

\newcommand{\Rp}{\,\Big)}
\newcommand{\Lp}{\Big(\,}

\def\inpr{\mathbin{\hbox to 6pt{\vrule height0.4pt width5pt depth0pt \kern-.4pt \vrule height6pt width0.4pt depth0pt\hss}}}

\usepackage{mathtools}
\def\rbinom#1#2{\ensuremath{\left(\kern-.3em\left(\genfrac{}{}{0pt}{}{#1}{#2}\right)\kern-.3em\right)}}

\newcommand{\as}{\mathrm{a.s.}}

\newcommand{\cpt}{{\mathrm{cpt}}}

\begin{document}

\title{A probabilistic computation of a Mehta integral} 

\author{Liviu I. Nicolaescu}
\thanks{}

\begin{abstract}  We use the Kac-Rice formula  to compute the Mehta integral describing  the normalization constant arising in the statistics of the Gaussian Orthogonal Ensemble. \end{abstract}

\date{ Last revised \today.}

\address{Department of Mathematics, University of Notre Dame, Notre Dame, IN 46556-4618.}
\email{nicolaescu.1@nd.edu}
\urladdr{\url{http://www.nd.edu/~lnicolae/}}

\maketitle

\tableofcontents
%\addtocontents{toc}{Contents}

\section{Introduction}   We denote by $\Sym(\bR^m)$ the space of real symmetric   $m\times m$  matrices. This is a Euclidean space with respect to the inner product $(A,B):=\tr(AB)$. This inner product is invariant with respect to the action of the orthogonal group $O(m)$ on $\Sym(\bR^m)$.  

 We define
\begin{equation}\label{xiija}
\ell_{ij},\omega_{ij}:\Sym(\bR^m)\to\bR,\;\;\ell_{ij}(A)=a_{ij},\;\;\omega_{ij}(A):= \begin{cases}
a_{ij}, & i=j,\\
\sqrt{2}a_{ij}, & i<j.
\end{cases}
\end{equation}
The collection $(\omega_{ij})_{i\leq j}$  defines linear  coordinates on $\Sym(\bR^m)$ that are orthonormal with respect  to the above inner product on $\Sym(\bR^m)$. The volume density induced by this metric is
\[
\vol\lb dA\rb:=\prod_{i\leq j} d\omega_{ij}= 2^{\frac{1}{2}\binom{m}{2}}\prod_{i\leq j} d\ell_{ij}.
\]
For any    real numbers  $u,v$ such that
\begin{equation}\label{pos_inv}
v>0, mu+2v>0,
\end{equation}
 we denote by $\eS_m^{u,v}$  the space  $\Sym(\bR^m)$     equipped with the centered Gaussian measure $\bGamma_{u,v}\lb dA\rb$ uniquely determined by the covariance equalities
 \begin{equation}\label{Smuv}
 \bE\lb \ell_{ij}(A)\ell_{k\ell}(A) \rb =  u\delta_{ij}\delta_{k\ell}+ v(\delta_{ik}\delta_{j\ell}+ \delta_{i\ell}\delta_{jk}),\;\;\forall 1\leq i,j,\;k,\ell\leq m.
 \end{equation}
 In particular we have
 \begin{equation}\label{inv_gauss1}
 \bE\lb \ell_{ii}^2\rb = u+2v,\;\;\bE\lb \ell_{ii}\ell_{jj}\rb =u,\;\;\;\bE\lb \ell_{ij}^2\rb=v,\;\;\forall 1\leq i\neq j\leq m,
 \end{equation}
while all other covariances are trivial.  The  ensemble  $\eS^{0,v}$ is    a rescaled version of  the Gaussian Orthogonal Ensemble  (GOE) and we will refer to it as $\GOE_{m}^{v}$.   The inequalities (\ref{pos_inv}) guarantee that the covariance form defined by (\ref{Smuv})  is positive definite so that $\bGamma_{u,v}$ is nondegenerate.

For $u>0$ the ensemble $\eS_m^{u,v}$ can be given an alternate description.   More precisely   a random $A\in \eS_m^{u,v}$ can be described as a sum
\[
A= B+ \ X\one_m,\;\;B\in \GOE_{m}^{v},\;\; X\in \bsN(0, u),\;\;\mbox{ $B$ and $X$ independent}.
\]
We write  this
\begin{equation}
\eS_m^{u,v} =\GOE_{m}^{v}\hat{+}N(0,u)\one_m,
\label{eq: smgoe}
\end{equation}
where $\hat{+}$ indicates a sum of \emph{independent} variables.

 In the special case  $\GOE_{m}^{v}$ we have $u=0$  and  
\begin{equation}
\bGamma_{0,v}\lb dA\rb=\frac{1}{(4\pi v)^{\frac{m(m+1)}{4}} } e^{-\frac{1}{4v}\tr A^2}  \vol\lb  dA\rb.
\label{eq: gov}
\end{equation}
Note that $\GOE_{m}^{1/2}$ corresponds to the Gaussian measure on $\Sym(\bR^m)$  canonically associated to the inner product $(A,B)=\tr(AB)$.

We have a   \emph{Weyl integration formula} \cite {AGZ} which states that if  $f: \Sym(\bR^m)\ra \bR$ is a measurable  function which  is invariant under  conjugation, then the     value $f(A)$ at $A\in\Sym(\bR^m)$ depends only on the eigenvalues $\lambda_1(A)\leq \cdots \leq \lambda_n(A)$ of $A$ and we  have
\begin{equation}
\bE_{\GOE_m^v}\lb f(X)\rb =\frac{1}{\bsZ_m(v)} \int_{\bR^m}  f(\lambda_1,\dotsc ,\lambda_m) \underbrace{\left(\prod_{1\leq i< j\leq m}|\lambda_i-\lambda_j| \right)\prod_{i=1}^m e^{-\frac{\lambda_i^2}{4v}} }_{=:Q_{m,v}(\lambda)}\; |d\lambda_1\cdots d\lambda_m|,
\label{eq: weyl}
\end{equation}
where the normalization constant $\bsZ_m(v)$ is defined  by
\[
{\bsZ_m(v)} =\int_{\bR^m}   \prod_{1\leq i< j\leq m}|\lambda_i-\lambda_j| \prod_{i=1}^m e^{-\frac{\lambda_i^2}{4v}} |d\lambda_1\cdots d\lambda_m|
\]
\[
=(2v)^{\frac{m(m+1)}{4}} \times \underbrace{\int_{\bR^m}   \prod_{1\leq i< j\leq m}|\lambda_i-\lambda_j| \prod_{i=1}^m e^{-\frac{\lambda_i^2}{2}} |d\lambda_1\cdots d\lambda_m|}_{=:\bsZ_m}.
\]
The  integral $\bsZ_m$ is usually referred  to as \emph{Mehta's integral}.\index{Mehta integral} Its value was first determined  in 1960 by  M. L. Mehta, \cite{MD}. Later Mehta  observed that this integral  was known earlier to N. G. de Brujin \cite{Bruj_Mehta}. It was subsequently observed  that Mehta's integral is  a limit of  the \emph{Selberg integrals}, \cite[Eq. (2.5.11)]{AGZ}, \cite[Sec. 4.7.1]{For}. More precisely, we have
\begin{equation}
\bsZ_m=(2\pi)^{\frac{m}{2}} \prod_{j=0}^{m-1}\frac{\Gamma(\frac{j+3}{2})}{\Gamma(3/2)}=2^{\frac{3m}{2}}\prod_{j=0}^{m-1}\Gamma\Lp\frac{j+3}{2}\Rp.
\label{eq: zm}
\end{equation}
The goal of this note is to provide a  probabilistic proof of (\ref{eq: zm}).

We will  determine $\bsZ_m$ inductively by  computing explicitly  the ratios $\frac{\bsZ_{m+1}}{\bsZ_m}$, $\forall m\geq 1$ and observing by immediate direct computation that  
\[
\bsZ_1=\int_\bR e^{-t^2/2} dt =(2\pi)^{1/2}.
\]
Here is the strategy. Any  symmetric $(m+1)\times (m+1)$ matrix  $A$ determines a  function on the unit sphere $S^m\subset \bR^{m+1}$
\[
\Phi_A: S^m\to \bR,\;\;\Phi_A(\bx)=\frac{1}{2}\lp A\bx,\bx\rp,
\]
where $(-,-)$ is the canonical  inner product on $\bR^{m+1}$.  When $A\in\GOE_{m+1}^{v}$, then  with probability  $1$  the matrix $A$ is simple and the Gaussian random  function  $\Phi_A$ is Morse. 

To the random matrix $A\in\GOE_{m+1}^{v}$ we can   associate two random  measures on $\bR$. The first is the \emph{spectral measure}
\[
\si_A=\sum_{\lambda\in \Spec(A)} \delta_\lambda,
\]
where $\delta_x$ denotes the Dirac measure on $\bR$ concentrated at $x$. The second one is the \emph{discriminant measure}
 \[
\bsD_A=\sum_{\nabla\Phi_A(\bx)=0}\delta_{2\Phi_A(x)}.
\]
The critical values of $2\Phi_A$ are precisely the eigenvalues of $A$ and the critical points are the unit eigenvectors of $A$. The function is Morse iff $A$ is simple, i.e., its eigenvalues are distinct.  In this case to each critical value of $A$ there corresponds exactly two critical points. With probability $1$ we have
\[
\bsD_A=2\bsi_A.
\]
Then for any Borel subset $C\subset \bR$ we have
\begin{equation}\label{bsDAC0}
\bE\lb \bsD_A\lb C\rb\rb=2\bE\lb \bsi_A\lb C\rb\rb.
\end{equation}
In particular
\begin{equation}\label{bsDAC1}
\bE\lb \bsD_A\lb \bR\rb\rb=2\bE\lb \bsi_A\lb \bR\rb\rb=2(m+1).
\end{equation}
Using   the Kac-Rice formula  we will be able to express  $\bE\lb \bsD_A\lb \bR\rb\rb$ as an \emph{explicit} multiple of the ratio $\frac{\bsZ_{m+1}}{\bsZ_m}$.

Here is  the structure of the paper. Section \ref{s: 2}   contains several  probabilistic digressions. The first one concerns   the expectation of the absolute value of characteristic polynomial  of a random matrix $A\in \GOE$.   The second  one  describes  a  version of the Kac-Rice formula  needed in the proof.  The last digression of this section is a well known  classical result commonly referred  to as the Gaussian regression formula.   We give a coordinate free description of this result not readily available in traditional probabilistic sources,  but very convenient to use in geometric applications.  Then last section  provides the  details of the strategy outlined above.

\section{Probabilistic digressions}\label{s: 2}
 For any positive integer $n$ we define the \emph{normalized} $1$-point correlation function $\rho_{n,v}(x)$ of $\GOE_{n}^{v}$ to be
\[
\rho_{n,v}(x)= \frac{1}{\bsZ_n(v)}\int_{\bR^{n-1}} Q_{n,v}(x,\lambda_2,\dotsc, \lambda_n) d\lambda_1\cdots d\lambda_n.
\]
For any Borel measurable function $f:\bR\ra \bR$  we have \cite[\S 4.4]{DG}
 \begin{equation}
 \frac{1}{n}\bE_{\GOE_{n}^{v}} \lb \tr f(X)\rb = \int_{\bR} f(\lambda) \rho_{n,v}(\lambda) d\lambda. 
 \label{eq: 1pcor}
 \end{equation}
The equality (\ref{eq: 1pcor}) characterizes $\rho_{n,v}$.   We want to draw attention to a confusing situation in the existing literature on the subject.  Some authors,  such as M. L. Mehta  \cite{Me}, define the  $1$-point correlation function $R_n(x)$ by the equality
\[
\bE_{\GOE_n^{1/2}} \lb \tr f(X)\rb = \int_{\bR} f(\lambda)  R_n(\lambda) d\lambda.
\]
The   expected value of the absolute value of the  determinant of     of a  random $A\in \GOE_m^v$ can be expressed neatly in terms of the correlation function $\rho_{m+1,v}$.   More precisely, we have the following result first observed by Y.V. Fyodorov \cite{Fy}.

\begin{lemma}\label{lemma: exp_det_GOE}  Suppose $v>0$. Then for any $c\in\bR$ we have
\begin{equation}\label{exp_det_GOE}
\bE_{\GOE_m^v} \lb |\det(A-c\one_m)|\rb =\frac{e^{\frac{c^2}{4v}}\bsZ_{m+1}(v)}{\bsZ_m(v)} \rho_{m+1,v}(c)=\frac{e^{\frac{c^2}{4v}}(2v)^{\frac{m+1}{2}}\bsZ_{m+1}}{\bsZ_m} \rho_{m+1,v}(c).
\end{equation}
\label{lemma: exp-det}
\end{lemma}

\begin{proof}Using  Weyl's integration formula we deduce
\[
\bE_{\GOE_m^v} \lb |\det(A-c\one_m)|\rb  =\frac{1}{\bsZ_m(v)}\int_{\bR^m} \prod_{i=1}^me^{-\frac{\lambda_i^2}{4v}}|c-\lambda_i|\prod_{i\leq j}|\lambda_i-\lambda_j| d\lambda_1\cdots d\lambda_m
\]
\[
=\frac{e^{\frac{c^2}{4v}}}{\bsZ_m(v)} \int_{\bR^m} e^{-\frac{c^2}{4v}}\prod_{i=1}^me^{-\frac{\lambda_i^2}{4v}}|c-\lambda_i|\prod_{i\leq j}|\lambda_i-\lambda_j| d\lambda_1\cdots d\lambda_m
\]
\[
= \frac{e^{\frac{c^2}{4v}}\bsZ_{m+1}(v)}{\bsZ_m(v)}\frac{1}{\bsZ_{m+1}(v)} \int_{\bR^m} Q_{m+1,v}(c,\lambda_1,\dotsc, \lambda_m) d\lambda_1\cdots d\lambda_m
\]
\[
=\frac{e^{\frac{c^2}{4v}}\bsZ_{m+1}(v)}{\bsZ_m(v)} \rho_{m+1,v}(c)=\frac{e^{\frac{c^2}{4v}}(2v)^{\frac{m+1}{2}}\bsZ_{m+1}}{\bsZ_m} \rho_{m+1,v}(c).
\]

\end{proof}

We will need a special version of the Kac-Rice formula. Let $(M,g)$ be a compact Riemann manifold. Denote by $\vol_g[-]$ the volume element on $M$ determined by $g$ and  $\nabla^g$  the Levi-Civita  connection of $g$.   If $F\in C^2(M)$, then we define the Hessian of $F$ at $\bp\in M$ to be the linear operator
 \[
 \Hess_F(\bp):T_\bp M\to T_\bp M,\;\; \Hess_F(\bp)X=\nabla^g_X\nabla
 F,
 \]
 where $\nabla^gF$ is the  metric gradient of $F$.
 
  Suppose that  $F:M \to\bR$ is a  Morse function. For any subset $S\subset M$ we denote  by $Z(S,dF)$ the number of critical points of $F$ inside $S$ and $B$ is an open subset  We denote by $\eD(F)$ the  \emph{discriminant set} \index{discriminant}  of $F$, i.e., the set of  critical values of $F$.  The \emph{discriminant measure} of $F$  is the pushforward
\[
\bsD_{F}=\sum_{t\in \bR} Z(F^{-1}(t), dF)\delta_t.
\]
The discriminant  measure is concentrated   on  $\eD(F)$. For $\vfi\in C^0_\cpt(\bR)$ we set
\[
\bsD_{F}\lb \vfi\rb:=\int_\bR \vfi d\bsD_{F}.
\]
When $F$ is random, $\bsD_{F}\lb \vfi\rb$ is a random variable. We have the following result \cite[Thm. 12.4.1]{AT}.

\begin{theorem}\label{th: KR} Suppose that  $F:M\to\bR$ is a $C^2$ Gaussian random function function satisfying the ampleness condition

\[
\tag{${\bf A}$}\label{tag: A}
\mbox{for any $\bp\in M$ the Gaussian  vector $F(\bp)\oplus dF(\bp)\in T^*_\bp M$ is nondegenerate.}
\]
We denote by $\bP_{F(p)}$ the probability distribution of the random variable $F(\bp)$ and by $p_{dF(p)}$ the probability \emph{density} of the Gaussian vector $dF(\bp)$. 

Then $F$ is $\as$ Morse and, for  any  function $\vfi\in C^0_\cpt(\bR)$ we have
 \begin{equation}\label{KR}
 \begin{split}
\bE\lb  \bsD_{F}[\vfi]\rb\hspace{12cm}\\
=\int_M\left(\int_\bR \bE\lb \lv \det \Hess_F(\bp)\rv\cond dF(\bp)=0,\,F(\bp)=t\rb \vfi(t)\bP_{F(\bp)}\lb dt\rb \right) p_{dF(\bp)}(0) \vol_g\lb d\bp\rb\\
=\int_M\bE\lb  \lv \det \Hess_F(\bp)\rv \vfi\lp F(\bp)\rp\cond dF(\bp)=0\rb.\hspace{5cm} 
\end{split}
\end{equation}
Above, $\bE\lb -\cond-\rb$ denotes appropriate conditional expectations.
\end{theorem}

When applying the Kac-Rice formula   we need to evaluate certain conditional expectations. In  the Gaussian case this is readily achieved using the classical Gaussian regression formula.  In the remainder of this section we describe this Gaussian regression in a form convenient  in geometric applications.

 Suppose that $\bsX$ and $\bsY$ are  finite dimensional vector spaces.    Consider two  random vectors 
\[
X:(\Omega,\eS,\bP)\to\bsX,\;\;  Y:(\Omega,\eS,\bP)\to\bsY, 
\]
 where $(\Omega,\eS,\bP))$ is a probability space. The mean or expectation of $X$ is the vector
 \[
 m(X)=\bE\lb X\rb=\int_\Omega X(\omega) \bP\lb d\omega\rb\in\bsX,
 \]
 whenever the integral is well defined.   The random vector $X$ is called centered if $m(X)=0$. 
 
 The \emph{covariance form} \index{covariance form}  of $Y$ and $X$  is  the bilinear form 
\[
\cov\lb Y,X\rb:\bsY^*\times \bsX^*\to\bR
\]
given by 
\[
\cov\lb Y,X\rb(\eta,\xi)=\cov \lb\lan \eta,Y\ran, \lan \xi,X\ran\rb,\;\;\forall \eta\in \bsY^*,\;\xi \in \bsX^*.
\]
If $\bsX$ and $\bsY$ are equipped with inner products $(-,-)_\bsX$ and respectively $(-,-)_\bsY$, then we can identify $\cov\lb Y,X\rb$ with a linear operator $C_{Y,X}:\bsX\to \bsY$. Concretely, if $(\be_i)_{i\in I}$  and $(\bsf_j)_{j\in J}$ are \underline{\emph{orthonormal}} bases of $\bsX$ and respectively $\bsY$, and we set $X_i:=(\be_i, X)_\bsX$, $Y_j:=(\bsf_j, Y)_\bsY$, then in these bases the operator $C_{Y,X}$ is described by matrix $(c_{ji})_{(j,i)\in J\times I}$, where $c_{ji}:=\cov\lb Y_j,X_i
\rb$.   Hence
\[
C_{Y,X}\be_i=\sum_jc_{ji}\bsf_j.
\]
We will refer to $C_{Y,X}$  as the \emph{correlator} of $Y$ with $X$.  Some times, for typographical reasons, we will use the alternate  notation\index{$\cor\lb Y,X\rb$} $\cor\lb Y,X\rb:=C_{Y,X}$.

The variance operator of $X$ is  $\var\lb X\rb:=C_{X,X}$. We say that $X$ is nondegenerate if its variance operator is invertible. Observe   that $C_{X,Y}:\bsY\to \bsX$  is the  adjoint of $C_{Y,X}$, $C_{X,Y}=C_{Y,X}^*$.

 If $\bsX$ and $\bsY$ are equipped with inner products, then $\bsX\oplus \bsY$ is equipped with the direct sum of these inner products and in this case $\var\lb X\oplus Y\rb: \bsX\oplus \bsY\to\bsX\oplus \bsY$ admits the block decomposition
\[
\Var\lb X\oplus Y\rb=\left[
\begin{array}{cc}
\var\lb X\rb & C_{X,Y}\\
C_{Y,X} &\var\lb Y\rb
\end{array}\right].
\]
The random vectors $X,Y$ are said to be \emph{jointly Gaussian} \index{jointly Gaussian} if the random vector $X\oplus Y$ is Gaussian.

\begin{proposition}[Gaussian regression formula]\label{prop: RF} \index{Gaussian regression}  \index{regression! formula}Suppose that $X,Y$ are  Gaussian vectors valued in the Euclidean spaces $\bsX$ and respectively  $\bsY$. Assume additionally that\index{formula! Gaussian regression}

\begin{enumerate} 

\item the random vectors $X, Y$ are  jointly Gaussian  and,

\item  $X$ is nondegenerate.

\end{enumerate}

 Define the \emph{regression operator} \index{regression! operator}
\begin{equation}\label{regr0}
R_{Y,X}:\bsX\to \bsY,\;\;R_{Y,X}:= C_{Y,X}\var[X]^{-1}
\end{equation}
Then the following hold.

\smallskip

\noindent (a) The conditional expectation $\bE\lb Y\cond X\rb$ is the Gaussian vector  described by the linear regression formula
 \begin{equation}\label{reg1}
 \bE\lb Y\cond X\rb= m(Y)-R_{Y,X} m(X) +R_{Y,X}X.
 \end{equation}
\noindent (b) For any $x\in X$
\[
\bE\lb Y\lv  X=x\rb= m(Y)-R_{Y,X} m(X) +R_{Y,X}x.
\]
(c) The random vector vector $Z=Y-\bE\lb Y\cond X\rb $ is Gaussian and  independent of $X$. It has  mean $0$ and variance operator
\begin{equation}\label{reg2}
\Delta_{Y,X} =\Var\lb Y\rb-D_{Y,X}:\bsY\to\bsY,\;\;D_{Y,X}=C_{Y,X}\Var[X]^{-1}C_{X,Y}.
\end{equation}
 Moreover, for any bounded measurable function $f:\bsY\to \bR$  and any $x\in \bsX$ we have
\begin{equation}\label{reg3}
\bE\lb f(Y)\lv X=x\rb =\bE\lb f\lp Z+ m(Y)-R_{Y,X}m(X)+ R_{Y,X}x\rp\rb.
\end{equation}
In particular, if $X$ and $Y$ are centered we have
\begin{equation}\label{reg3a}
\bE\lb f(Y)\lv X=x\rb =\bE\lb f\lp Z+ R_{Y,X}x\rp\rb.
\end{equation}
\end{proposition}
For a proof we  refer to \cite[Prop. 2.1]{AzWs}.

  \section{The computation of the Mehta integral} As explained in the introduction,  when $A$ runs in the Gaussian ensemble $\GOE_{m+1}^{v}$ we obtain a Gaussian function
\[
\Phi=\Phi_A: S^m\to\bR.
\]
This function  is invariant under the natural $O(m+1)$-action on $S^m$.  

\begin{lemma}\label{lemma: quad_Morse} The Gaussian function $\Phi_A$ is $\as$  Morse.
\end{lemma}

\begin{proof} It suffices to show that the Gaussian section $\nabla\Phi_A$ of $TS^m$ i satisfies the ampleness condition (\ref{tag: A}).  Let $\bx\in S^m$. If  $\Proj_\bx:\bR^{m+1}\to\bR^{m+1}$ the orthogonal projection onto $T_\bx S^m$, then
\[
\nabla\Phi_A(\bx)=\Proj_\bx A\bx= A\bx- (A\bx,\bx)\bx.
\]
The map
\[
\Sym_{m+1}(\bR)\ni A\mapsto A\bx\in \bR^{m+1}
\]
is onto   and thus the map
\[
\Sym_{m+1}(\bR)\ni A\mapsto \Proj_\bx A\bx\in T_\bx S^m
\]
is also onto, thus proving  that  the gradient $\nabla\Phi_A(\bx)$ is nondegenerate since the Gaussian ensemble $\GOE_{m+1,v}$ is nondegenerate.
\end{proof}

The spectral measure of $A$ is
\[
\bsi_{A}:=\sum_{\lambda\in \Spec(A)} \mult(\lambda)\delta_\lambda.
\]
The discriminant  measure  of $\Phi_A$ is
\[
\bsD_A=\sum_{\nabla\Phi_A(\bx)=0}\delta_{2\Phi_A(x)}.
\]
 With probability $1$ we have $\bsD_A=2\bsi_A$. Then for any Borel subset $C\subset \bR$ we have
\begin{equation}\label{bsDAC}
\frac{1}{(m+1)}\bE\lb \bsD_A\lb C\rb\rb=\frac{2}{m+1}\bE\lb \bsi_A\lb C\rb\rb=2\int_C\rho_{m+1,v}(\lambda) d\lambda.
\end{equation}
Using the Kac-Rice formula (\ref{KR})  we will give an alternate description of the left-hand-side of the above equality.  We will need to describe explictly the integrand in this formula. 

For $\bx\in S^m$ we denote by $\Hess_A(\bx)$ the Hessian of $\Phi_A$ at $\bx$ viewed  as  a symmetric operator $T_\bx S^m\to T_\bx S^m$.

Denote by  $(x^0,x^1,\dotsc, x^m)$ the canonical Euclidean coordinates on $\bR^{m+1}$.  Since $\Phi_A$ is $O(m+1)$ invariant, the distribution of $\Hess_A(\bx)$ is independent of $\bx$ so it suffices to determine it at any point of our choosing.  Suppose that $\bx$ is the north pole 
\[
\bx=\bn=(1,0,\dotsc, 0)\in \bR^{m+1}
\]
Then $T_{\bn}S^m=\{x^0=0\}$  and $\bx_*:=\lp x^1,\dotsc, x^m)$ are orthonormal coordinates on $T_{\bn}S^m$.   The coordinates $\bx_*$ also define  local coordinates on $S^m$. More precisely,  the upper  hemisphere
\[
S^m_+:=\lbr \bx\in S^m;\;\;x^0>0 \rbr
\]
admits the parametrization
\[
\bx^*\mapsto \bx\lp\bx_*\rp= \lp x^0(\bx_*) ,\bx^*\rp\in S^m,\;\;x^0(\bx_*)=\sqrt{1-\Vert\bx_*\Vert^2}.
\]
The round metric  on $S^m$ satisfies
\begin{equation}\label{near_euclidean}
g_{ij}=\delta_{ij}+ O\lp\Vert\bx_*\Vert^2\rp\;\;\mbox{near $\bn$}.
\end{equation}
On the upper hemisphere we will view $\Phi_A$ as a function of $\bx_*$.  

If $A=(a_{ij})_{0\leq i,j\leq m}$, then  in the coordinates $\bx_*$ we have
  \[
  \Phi_A(\bx)=\frac{1}{2}a_{00}\lp 1-\Vert \bx_*\Vert^2\rp+ \frac{1}{2}\sum_{j=1}^m a_{jj} (x^j)^2+\sum_{0\leq j<k\leq m} a_{jk}x^jx^k,
  \]
  \[
  =\frac{1}{2}a_{00}^2+  \frac{1}{2}\sum_{j=1}^m \lp a_{jj}-a_{00}\rp (x^j)^2+\sum_{0\leq j<k\leq m} a_{jk}x^jx^k,
  \]
   \[
  \nabla\Phi_A(\bn)=d\Phi_A(\bx_*)\vert_{\bx_*=0}=\sum_{j=1}^ma_{0j} dx^j.
  \]
  Since $A\in \GOE_{m+1}^{v}$, covariance kernel  of $\Phi_A$ is
  \[
  \eK_A(\bn,\bx)=\bE\lb \Phi_A(\bn)\Phi_A(\bx)\rb= \frac{1}{4} \lp 1-\Vert\bx_*\Vert^2\rp\bE\lb a_{00}^2\rb= \frac{v}{2} \lp 1-\Vert\bx_*\Vert^2\rp.
  \]
Denote by $A_*$ the $m\times m$ matrix $A_*=(a_{ij})_{1\leq i\leq m}$. Note that $A_*\in \GOE_m^v$. Using (\ref{near_euclidean})
\[
  \Hess_A(\bn)= A_*-a_{00}\one_m.
  \]
 Since $a_{00}$ is independent of $A_*$  we deduce  from (\ref{eq: smgoe}) that $\Hess_A(0)\in \eS^{2v,v}_m$, where $\eS_m^{u,v}$ is the $O(m)$-invariant Gaussian ensemble defined by (\ref{Smuv}).  If we set
 \[
 L_{ij}=\ell_{ij}\lp \Hess_A(\bn)\rp,\;\;  \Omega_{iij}=\omega_{ij}\lp \Hess_A(\bn)\rp,
 \]
 where $\ell_{ij}$ and $\omega_{ij}$ are defined by (\ref{xiija}), then
 \[
  \bE\lb L_{ij}L_{k\ell}(A) \rb =  2v\delta_{ij}\delta_{k\ell}+ v(\delta_{ik}\delta_{j\ell}+ \delta_{i\ell}\delta_{jk}),\;\;\forall 1\leq i,j,.k,\ell\leq m.
\]
  Note that $\nabla\Phi_A(\bn)=(a_{01},\dotsc, a_{0n})$ is independent of $\Phi_A(\bn)$ and of  $\Hess_A(\bn)$ and since $A\in \GOE_{m+1}^v$ we  deduce from (\ref{Smuv}) that
 \begin{equation}\label{grad_iso_12}
 \var\lb \Phi_A \rb=v\one_m.
 \end{equation}
Moreover the  correlator $\cor\lb \Hess_A(\bn),\Phi_A(\bn\rb:\bR\to\Sym_m(\bR)$ is  given by
\[
\bR\ni x\mapsto -x\one_m\in\Sym_m(\bR).
\]
Set
\[
W=\left[
\begin{array}{c}
\Phi_A(\bn)\\
\nabla \Phi_A(\bn)
\end{array}
\right]= \left[
\begin{array}{c}
\frac{1}{2}a_{00}\\
a_{01}\\
\vdots\\
a_{0m}
\end{array}
\right].
\]
Note that
\[
\var\lb W\rb= \diag\lp \frac{v}{2}, \underbrace{2v,\dotsc, 2v}_m\rp.
\]
Denote by $\xbar{\Hess}_A(\bn)$ the random symmetric matrix with variance given by the regression formula
\[
\var\lb \xbar{\Hess}_A(\bn)\rb =\var\lb \Hess_A(\bn)\rb -\cor\lb \Hess_A(\bn),W\rb \var\lb W\rb^{-1}\cor\lb W, \Hess_A(\bn)\rb.
\]
Set
\[
\xbar{L}_{ij}=\ell_{ij}\lp \xbar{\Hess}_A(\bn)\rp,\;\;\xbar{\Omega}_{ij}:=\omega_{ij}\lp \xbar{\Hess}_A(\bn)\rp,
\]
 and 
 \[
 C_{ij|k}:=\cov\lb \Omega_{ij}, W_k\rb,\;\;1\leq i\leq j\leq m,\;\;0\leq k\leq m.
 \]
 Note that
 \[
 C_{ij|k}=0,\;\;\forall i,j,\;\;\forall k>0,\;\;C_{ij|k}=0,\;\;\forall  i<j,\;\;\forall k\geq 0,
 \]
 and
 \[
 C_{ii|0}= \frac{1}{2}\bE\lb (a_{ii}-a_{00})a_{00}\rb=-\frac{1}{2}\bE\lb a_{00}^2\rb=-v.
 \]
 If we write
\[
\var\lb W\rb^{-1}= \lp t_{ab}\rp_{0\leq a,b\leq m},
\]
then
\[
\bE\lb \xbar{\Omega}_{ij}\xbar{\Omega}_{k\ell}\rb = \bE\lb \Omega_{ij}(\bx)\Omega_{k\ell}(\bx)\rb-\sum_{a,b=0}^m C_{ij| a}t_{ab}C_{k\ell |b} = \bE\lb \Omega_{ij}(\bx)\Omega_{k\ell}(\bx)\rb-\frac{2}{v}C_{ij|0}C_{k\ell|0}
\]
  For $i\neq j$
\[
\bE\lb  \xbar{\Omega}_{ii} \xbar{\Omega}_{jj}\rb=  \bE\lb \Omega_{ii}(\bx)\Omega_{jj}(\bx)\rb-2v=0
\]
\[
\bE\lb  \xbar{\Omega}_{ii}^2\rb =2v,
\]
\[
\bE\lb \xbar{\Omega}_{ij}\xbar{\Omega}_{k\ell}\rb = \bE\lb \Omega_{ij}(\bx)\Omega_{k\ell}(\bx)\rb,\;\;\forall 1\leq k\leq\ell.
\]
We deduce that $\xbar{\Hess}_A\in \GOE_{m}^{v}$ so  $\xbar{\Hess}_A$ has the same distribution as $A$.

 The regression operator
\[
R_{\Hess_A, W}=\cor\lb \Hess_A,W\rb \var\lb W^{-1}\rb:\bR^{m+1}\to\Sym_{m}(\bR)
\]
is
\[
\left[
\begin{array}{c}
w_0\\
w_1\\
\vdots\\
w_m
\end{array}
\right]\mapsto -2w_0\one_m.
\]
Using the regression formula we deduce
 \[
  \bE\lb |\Hess_A(\bn)|\cond W(\bn)=(t/2,0)\rb=\bE_{\GOE_{m,v}}\lb  \vert\det (A-v t)\vert\rb. 
\]
Since $2\Phi_A(\bn)=a_{00}$ is Gaussian with  variance $2v$, we deduce. from the Kac-Rice formula (\ref{KR})  that for  any Borel subset $C\subset \bR$ we have
\[
\bE\lb \bsD_A\lb C\rb\rb=\int_C \rho_A(t)\bgamma_{2v}\lb dt\rb,
\]
where
\begin{equation}\label{e_discr_sph}
\begin{split}
\rho_A(t)=\int_{S^m} \bE\lb |\Hess_A(\bx)|\cond  2\Phi_A(\bx)= t,\;\nabla\Phi_A(\bx)=0\rb  p_{\nabla\Phi_A(\bx)}(0) d\bx\\
\stackrel{(\ref{grad_iso_12})}{=}\lp 2\pi v\rp^{-m/2} \int_{S^m} \bE\lb |\Hess_A(\bx)|\cond W(\bx)=(t/2,0)\rb d\bx
\\
= \lp 2\pi v\rp^{-m/2} \bE\lb |\Hess_A(\bn)|\cond W(\bn)=(t/2,0)\rb \vol\lb S^m\rb\\
=\lp 2\pi v\rp^{-m/2} \vol\lb S^m\rb\bE_{\GOE_m^v}\lb |\det (A-vt)|\rb.
\end{split}
\end{equation}
In  Lemma \ref{lemma: exp_det_GOE}  we showed
\[
\bE_{\GOE_{m,v}}\lb \det (A-vt)|\rb=(2v)^{\frac{m+1}{2}} e^{\frac{v^2t^2}{4v}}\frac{\bsZ_{m+1}}{\bsZ_m} \rho_{m+1, v}(vt).
\]
Assume now that  $v=1$.  
\[
\bE_{\GOE_m^1}\lb \det (A-t)|\rb = e^{\frac{t^2}{4}}2^{\frac{m+1}{2}}\pi^{-1/2}\frac{\bsZ_{m+1}}{\bsZ_m} \rho_{m+1, 1}(t)
\]
Since $\bgamma_{2v}\lb dt\rb=e^{-\frac{t^2}{4}}\frac{dt}{\sqrt{4\pi}}$ we deduce
\[
\bE\lb \bsD_A\lb C\rb\rb =\lp 2\pi \rp^{-m/2} 2^{\frac{m+1}{2}}\vol\lb S^m\rb\frac{\bsZ_{m+1}}{\bsZ_m}\int_C\rho_{m+1,1}(t)  \frac{dt}{\sqrt{4\pi}}.
\]
On the other hand,  we deduce from  (\ref{bsDAC}) that
\[
\frac{1}{(m+1)}\bE\lb \bsD_A\lb C\rb\rb = 2\int_C\rho_{m+1,1}(t)  dt,
\]
so that
\[
\frac{ \lp 2\pi \rp^{-m/2} 2^{\frac{m+1}{2}}\vol\lb S^m\rb}{(m+1)} \frac{\bsZ_{m+1}}{\bsZ_m}(4\pi)^{-1/2}=2.
\]
Using the fact that
\[
\frac{\vol\lb S^m\rb }{m+1}= \frac{\pi^{\frac{m+1}{2}} }{\Gamma(\frac{m+3}{2}) }
\]
we deduce 
\[
\frac{\bsZ_{m+1}}{\bsZ_m}=\frac{\Gamma(\frac{m+3}{2}) }{\pi^{\frac{m+1}{2}} }\cdot \frac{2(2\pi)^{m/2}(4\pi)^{1/2}}{2^{\frac{m+1}{2}}}=2^{3/3}\Gamma\Lp\frac{m+3}{2}\Rp.
\]
Note that
\[
\bsZ_1=\int_\bR e^{-t^2/2} dt =(2\pi)^{1/2}.
\]
We deduce immediately the  equality   (\ref{eq: zm})
\[
\bsZ_m=\bsZ_1\prod_{j=1}^{m-1}\frac{\bsZ_{j+1}}{\bsZ_j}=2^{\frac{3m}{2}}\prod_{j=0}^{m-1}\Gamma\Lp \frac{j+3}{2}\Rp.
\]

\end{document}